\documentclass[12pt,reqno]{amsart}
\usepackage{amssymb}
\usepackage{graphicx}
\usepackage{xcolor}
\usepackage{longtable}
\usepackage{tikz}
\usepackage{tikz-cd}
\usepackage[all]{xy}
\usepackage{caption}
\usepackage{caption} 
\DeclareFontFamily{U}{mathb}{\hyphenchar\font45}
\DeclareFontShape{U}{mathb}{m}{n}{
      <5> <6> <7> <8> <9> <10> gen * mathb
      <10.95> mathb10 <12> <14.4> <17.28> <20.74> <24.88> mathb12
      }{}
\DeclareSymbolFont{mathb}{U}{mathb}{m}{n}
\DeclareMathSymbol{\righttoleftarrow}{3}{mathb}{"FD}

\newcommand{\actsfromright}{\righttoleftarrow}

\theoremstyle{plain}
\newtheorem{prop}{Proposition}
\newtheorem{theo}[prop]{Theorem}

\theoremstyle{remark}

\theoremstyle{definition}
\newtheorem{rema}[prop]{Remark}

\newtheorem{exam}[prop]{Example}
\numberwithin{equation}{section}
\def\Am{{\mathrm{Am}}}
\def\fA{{\mathfrak A}}

\def\fD{{\mathfrak D}}

\def\fS{{\mathfrak S}}

\def\fS{{\mathfrak S}}

\def\bG{{\mathbb G}}
\def\bP{{\mathbb P}}

\def\bZ{{\mathbb Z}}

\def\rH{{\mathrm H}}

\def\lra{\longrightarrow}

\def\bF{{\mathbb F}}

\def\Br{\mathrm{Br}}

\def\Pic{\mathrm{Pic}}

\def\Aut{\mathrm{Aut}}

\def\PSL{\mathsf{PSL}}

\def\Hom{\mathrm{Hom}}

\def\lim{\mathrm{lim}}

\def\Ker{\mathrm{Ker}}

\begin{document}
\title[Equivariant unirationality]{Cohomological obstructions to equivariant unirationality}

\author[Yuri Tschinkel]{Yuri Tschinkel}
\address{
  Courant Institute,
  251 Mercer Street,
  New York, NY 10012, USA
}
\email{tschinkel@cims.nyu.edu}

\address{Simons Foundation\\
160 Fifth Avenue\\
New York, NY 10010\\
USA}

\author{Zhijia Zhang}
\address{
  Courant Institute,
  251 Mercer Street,
  New York, NY 10012, USA 
}
\email{zz1753@nyu.edu}

\date{\today}

\begin{abstract}
We study cohomological obstructions to equivariant unirationality, with special regard to actions of finite groups on del Pezzo surfaces and Kummer quartic double solids.  
\end{abstract}

\maketitle

\section{Introduction}
\label{sec.intro}

Let $X$ be a smooth projective 
rational variety with a generically free regular action of a finite group $G$. We say that $X$ is {\em $G$-unirational}, respectively, {\em projectively $G$-unirational}, if there exists a dominant $G$-equivariant rational map
$$
\bP(V)\dashrightarrow X,
$$
where $V$ is a representation of $G$, respectively, of a projective representation of $G$. In the analogy between geometry over nonclosed fields and equivariant geometry, this should be viewed as being dominated by projective space, i.e., unirationality, versus being dominated by a Brauer-Severi variety. 

An obvious obstruction to rationality of $X$ over a field $k$ is the absence of $k$-rational points -- a birational invariant. In the equivariant context, 
existence of $G$-fixed points is not an equivariant birational invariant -- only the existence of fixed points upon restriction to abelian subgroups is; we refer to this as Condition {\bf (A)}.       

In general, unirationality and $G$-unirationality are difficult to establish or exclude, when the obvious obstructions, such as the absence of $k$-rational points, respectively, failure of Condition {\bf (A)}, vanish. 
In this note, we explore new cohomological obstructions to (projective) unirationality, and apply them to del Pezzo surfaces and Fano threefolds. 
The obstructions arise from considerations of the $G$-action on the Picard group $\Pic(X)$. 
Concretely, the Leray spectral sequence yields homomorphisms in group cohomology
$$
\rH^{j-2}(G,\Pic(X))\stackrel{\delta_j}{\lra} \rH^j(G,k^\times), \quad 
j=2,3, 
$$
where $G$ acts trivially on $k^\times$. The images of $\delta_j$ are stable birational invariants. Moreover, they vanish in the presence of $G$-fixed points and 
for $G$-unirational $X$, see Section~\ref{sect:coho}. 

Similar homomorphisms exist in the framework of birational geometry over nonclosed fields $k$, for Galois cohomology. However, their role in that context is limited: the images are trivial in presence of $k$-points, which is a (stable) birational invariant and an obvious necessary condition for $k$-unirationality. In the equivariant context, the invariants are quite subtle and informative -- this highlights a stark difference between birational geometry over $k$ and over the stack $BG$.

The main result of this paper is a classification of generically free regular actions of finite groups $G$ on smooth projective $X$ such that 
\begin{itemize}
    \item $X^A\neq \emptyset$, for all abelian $A\subseteq G$, and
    \item $\mathrm{Am}^3(X,H)\neq 0$, for 
   some $H\subseteq G$, 
\end{itemize}
when $X$ is a del Pezzo surface 
(see Theorem~\ref{theo:dp}) or a Kummer
quartic double solid, i.e., a 
double cover of $\bP^3$, ramified in a Kummer surface arising from the  
Jacobian of a genus 2 curve with maximal automorphisms
(see Theorem~\ref{thm:kummer}). 
More precisely, in Theorem~\ref{theo:dp}, we find all pairs $(X,G)$ where $X$ is a del Pezzo surface and $G\subset\Aut(X)$  a finite group containing a subgroup $H$
with nontrivial  $\Am^3(X,H)$. For some of these groups acting on del Pezzo surfaces of degree 2 one has $X^A\neq \emptyset$, for all abelian subgroups $A\subset G$. This shows that the analog of \cite[Theorem 1.4]{Duncan}
fails for del Pezzo surfaces of degree 2. In all cases, we find a $\mathrm{Q}_8$, the quaternion group of order 8, 
that leads to the nontriviality of the  obstruction. As a corollary, such varieties are not $G$-unirational, nor projectively $G$-unirational.

\

\noindent
{\bf Acknowledgments:} 
We are grateful to Andrew Kresch for his interest and suggestions. 
The first author was partially supported by NSF grant 2301983.
We are grateful to the referee for an exceptionally careful reading and many helpful suggestions.

\section{Cohomological obstructions}
\label{sect:coho}

Let $k$ be an algebraically closed field of characteristic zero and 
$X$ a smooth projective rational variety over $k$, with a regular action of a finite group $G$. The Leray spectral sequence for $G$-actions yields an exact sequence (see, e.g., \cite[Section 3]{KT-dp}):
\begin{align}
\begin{split}
\label{eqn:BrXG}
0&\to  \Hom(G,k^\times)\to \Pic(X,G)\to 
\Pic(X)^G \stackrel{\delta_2}{\lra}  \rH^2(G,k^\times)\\
&\qquad\quad \!\stackrel{\gamma}{\lra} \Br([X/G])\stackrel{\beta}{\lra} \rH^1(G,\Pic(X))\stackrel{\delta_3}{\lra}  \rH^3(G,k^\times),
\end{split}
\end{align}
where $\Pic(X,G)$ is the group of isomorphism classes of $G$-linearized line bundles on $X$, 
and $\Br([X/G])$ is the Brauer group of the quotient stack. 
This gives rise to the following invariants:
\begin{itemize}
    \item $\mathrm{Am}^2(X,G):=\mathrm{Im}(\delta_2)$, the {\em Amitsur group}, and 
    \item  $\mathrm{Am}^3(X,G):=\mathrm{Im}(\delta_3)$.  
\end{itemize}

The {\em Amitsur} group 
$$
\mathrm{Am}(X,G)=\mathrm{Am}^2(X,G),
$$
defined in \cite[Section 6]{BC-finite} as the image of $G$-invariant divisor classes 
$$
\Pic(X)^G\stackrel{\delta_2}{\longrightarrow} \rH^2(G,k^\times), 
$$
is a stable $G$-birational invariant. The same holds for the higher Amitsur group $\mathrm{Am}^3(X,G)$, see \cite[Section 3]{KT-dp}. 

We note in passing that basic results in group cohomology allow to reduce the study of these invariants to $p$-Sylow subgroups: vanishing of $\mathrm{Am}^j(X,H)$ for $p$-Sylow subgroups  $H$ of $G$, for all $p$, implies the vanishing for all subgroups of $G$.
Furthermore, both groups vanish when $G$ has a fixed point on $X$, by functoriality. They vanish for $G$-varieties 
which are stably linearizable, or when $G$ is cyclic.  
We record the following refinement: 

\begin{prop}
    \label{prop:obstr}
Let $Y\to X$ be a $G$-equivariant morphism of smooth projective varieties
with regular $G$-actions. Then 
$$
\mathrm{Am}^j(X,G)\subseteq \mathrm{Am}^j(Y,G), \quad j=2,3.
$$
\end{prop}

\begin{proof}
Containment follows from functoriality of the Leray spectral sequence:
\[
\xymatrix{
\Pic(X)^G\ar[r]^(0.45){\delta_2}\ar[d] & \rH^2(G,k^\times)\ar@{=}[d] && \rH^1(G,\Pic(X)) \ar[r]^(0.55){\delta_3}\ar[d] & \rH^3(G,k^\times)\ar@{=}[d] \\
\Pic(Y)^G\ar[r]^(0.45){\delta_2} & \rH^2(G,k^\times) && \rH^1(G,\Pic(Y)) \ar[r]^(0.55){\delta_3} & \rH^3(G,k^\times)
}
\]
\end{proof}

We are now in the position to formulate necessary conditions for
\begin{itemize} 
\item $G$-unirationality: 
\begin{itemize}
    \item Condition {\bf (A)}: for all abelian $H\subseteq G$, one has $X^H\neq \emptyset$.
    \item Amitsur: for all $H\subseteq G$, one has $\mathrm{Am}^j(X,H)=0$, $j=2,3$. 
\end{itemize}
\item projective $G$-unirationality: 
\begin{itemize}
    \item Amitsur: for all $H\subseteq G$, one has that 
    $\mathrm{Am}^2(X,H)$ is cyclic.
    \item Amitsur: for all $H\subseteq G$, one has 
    $\mathrm{Am}^3(X,H)=0$. 
\end{itemize}
\end{itemize}

Indeed, for {\em linear} actions, we have
$$
\mathrm{Am}^2(X,G)=0, 
$$
since, by definition, the generator of $\Pic(\bP(V))$ is $G$-linearized. For {\em projectively linear} actions, the Amitsur group is cyclic. 
Furthermore, for linear and projectively linear actions, we have 
$$
\rH^1(G,\Pic(\bP(V)))=0,  
$$
which implies the vanishing
$$
\mathrm{Am}^3(X,G)=0.  
$$

We recall  the definition of the {\em Bogomolov multiplier} of a finite group:
$$
\mathrm{B}^2(G):=\Ker\left(\rH^2(G,k^\times)\to \bigoplus_{A}\,  \rH^2(A,k^\times)\right), 
$$
where $A$ runs over all abelian subgroups of $G$. 
This invariant emerged in the study of Noether's problem; the main result is that it equals the unramified Brauer group of the quotient $V/G$, where $V$ is a faithful representation of $G$ \cite{Bog-linear}.
More generally, one may consider higher-degree versions
$$
\mathrm{B}^n(G):=\Ker\left(\rH^n(G,k^\times)\to \bigoplus_{A}\,  \rH^n(A,k^\times)\right),  \quad n\ge 2. 
$$
For $n=3$, it is easy to obtain nonvanishing, see the tables below. 

\begingroup
\tiny
\setlength{\LTcapwidth}{\textwidth}
\captionsetup{font=tiny}
\begin{longtable}{|c|c|c|}
  \caption{Bogomolov multipliers}\label{table:bg1}\\
\hline
  GapID&$G$ & $\mathrm{B}^3(G)$\\\hline
  {\tt (8,4)}&$\mathrm{Q}_8$&$\bZ/2$\\\hline
  \hline
{\tt (16,9)}&$\mathrm{Q}_{16}$&$\bZ/2$\\\hline
{\tt (16,12)}&$C_2\times\mathrm{Q}_{8}$&$(\bZ/2)^3$\\\hline
{\tt (16,13)}&$\fD_4:C_2$&$\bZ/2$\\\hline
\hline
  {\tt (81,3)}&$C_3^2:C_9$&$\bZ/3$\\\hline
  {\tt (81,8)}&$\mathrm{He}_3.C_3$&$\bZ/3$\\\hline
  {\tt (81,10)}&$C_3.\mathrm{He}_3$&$(\bZ/3)^2$\\\hline
  {\tt (81,13)}&$C_3\times C_9:C_3$&$\bZ/3$\\\hline
  {\tt (81,14)}&$C_3.C_3^3$&$\bZ/3$\\\hline
 \end{longtable}
\endgroup
Even more generally, one may consider
$$
\mathrm B^n(G,M):=\Ker\left(\rH^n(G,M)\to \bigoplus_{A}\,  \rH^n(A,M)\right),  \quad n\ge 2, 
$$
for an arbitrary $G$-module $M$; as explained in \cite{KT-uni-toric}, the group 
$$
\mathrm B^2(G, \Pic(X)^\vee\otimes k^\times)
$$
also receives an obstruction to $G$-unirationality, when the action satisfies Condition {\bf (A)}.

The following proposition clarifies the connection between
Bogomolov multipliers and Amitsur invariants:

\begin{prop}
Let $X$ be a smooth projective variety with a regular action of a finite group $G$, satisfying Condition {\bf (A)}. Then 
$$
\mathrm{Am}^j(X,G)\subseteq \mathrm{B}^j(G), \quad j=2,3. 
$$
\end{prop}

In the following sections we present examples with obstructions to $G$-unirationality, based on the nonvanishing $\mathrm{Am}^3(X,G)$, for $G=\mathrm{Q}_8$, the smallest group with nonvanishing $\mathrm B^3(G)$.

We now recall the general formalism for the computation of $\mathrm{Am}^3(X,G)$, 
from \cite[\S 4]{KT-Brauer}: Choose an appropriate Zariski open subset $U\subset X$, with boundary divisors $D_{\alpha}$, $\alpha \in \mathcal A$, generating $\Pic(X)$. We have an exact sequence
\begin{equation}
    \label{eqn:seq}
0\to R\to \bigoplus_{\alpha\in \mathcal A} D_\alpha \to \Pic(X)\to 0, 
\end{equation}
where $R$ is the module of relations between the $D_{\alpha}$. 
The diagram of exact sequences 
\[
\xymatrix@C=12pt@R=15pt{
&&
\rH^2(G, \mathbb G_m(U))\ar[d] & \\
0 \ar[r] &
\rH^1(G,\Pic(X)) \ar[r]\ar[dr]_{\delta_3} &
\rH^2(G, R) \ar[r] \ar[d] &
\rH^2(G,\bigoplus_{\alpha \in \mathcal A} D_{\alpha}) \\
&&\rH^3(G,k^\times)&
}
\]

\

\noindent
allows to compute $\delta_3$, where the vertical sequence arises from the exact sequence 
$$
1\to k^\times\to \bG_m(U)\to R\to 0.
$$

\section{Del Pezzo surfaces}
\label{sect:dp}

In this section, we study cohomological obstructions to (projective) unirationality of regular, generically free, $G$-actions on smooth del Pezzo surfaces $X$, over algebraically closed fields $k$ of characteristic zero; by convention, the $G$-action on $X$ is from the right.  
Let 
$$
\deg(X):=(-K_X)^2 \in [9,\ldots,1] 
$$
be the degree of $X$. 
We summarize known results: 
\begin{itemize}
\item $G$-actions are known, in principle \cite{DI}. 
\item The groups 
$
\mathrm{Am}^2(X,G)
$ 
have been determined \cite[Proposition 6.7]{BC-finite}; these are trivial when  $\mathrm{rk}\,\Pic(X)^G=1$ and $\deg(X)\le 6$. 
\item 
The groups $\rH^1(G,\Pic(X))$ have been 
determined, starting with \cite{manin}, \cite{urabe}; these are trivial when $\deg(X)\ge 5$. 
There are algorithms to compute them when $G$ is cyclic, in 
\cite{BogPro}, \cite{Shinder}, and also for general $G$, in \cite{KT-dp}. 
E.g., for a del Pezzo surface with an involution fixing a (necessarily unique) smooth curve of genus $g$, one has
$$
\rH^1(G,\Pic(X))=(\bZ/2)^{2g}. 
$$
\item There is an algorithm to compute $\Br([X/G])$, in \cite{KT-dp} and \cite{KT-Brauer}.
It is based on the stabilizer stratification of the $G$-action.
When $G$ is abelian and $\mathrm{rk}\,\Pic(X)^G=1$, all possibilities of $\Br([X/G])$ have been determined in \cite{PZ}.
\item 
When $X$ is a del Pezzo surface of degree 1, one always has $X^G\neq \emptyset$, in particular, 
$$
\mathrm{Am}^j(X,H) =0, \quad \forall H\subseteq G, \quad j=2,3.
$$
\item $G$-unirationality of del Pezzo surfaces of degree $\ge 3$ has been settled in \cite{Duncan}, it is equivalent to Condition {\bf (A)}.  
\end{itemize}
We complement these results by analyzing
    $$
    \mathrm{Am}^3(X,G),
    $$
for all  del Pezzo surfaces $X$.    
In particular, we classify all actions with Amitsur obstructions to (projective) $G$-unirationality. 

In the following, we rely on (refinements of) tables of possible actions, going back to \cite{DI}. We focus on $p$-groups for a prime $p$, and del Pezzo surfaces of degree $4,3,2$. As already noted, if a class is nontrivial then it is nontrivial upon restriction to some $p$-Sylow subgroup. Furthermore, since $G$ is a subgroup of the Weyl group of $\mathsf D_5, \mathsf E_6, \mathsf E_7$, respectively, 
it suffices to consider primes dividing the order of these groups, i.e., 2,3,5,7, but actually, only $p=2$ or $3$, as explained below.  
The tables below list pairs $(X,G)$ satisfying the following 

\ 

{\em Condition {\bf (N)}}: 
\begin{itemize}
\item $X$ is $G$-minimal, i.e., there are no equivariant contractions, 
\item $X^G = \emptyset$,
\item $\rH^1(G,\Pic(X))\neq 0$, 
\item the $G$-action on $X$ fails Condition {\bf(A)}, when $\deg(X)\ge 3$. 
\end{itemize}
These properties are preserved in smooth families. 
Condition {\bf (N)} is necessary for nontriviality of $\Am^3(X,G)$, for $G$-minimal $X$.

\subsection*{Del Pezzo surfaces of degree 4} 
There are 5 types of actions, distinguished in \cite[Section 6]{DI}.
Type (I) and (II) have moduli, and the other types are rigid. 
A $p$-group action on $X$ is not $G$-minimal if $p\ne 2$, so we focus on $2$-groups.
All 2-group actions can be realized on types (I), (II), and (III): 
\begin{itemize}
\item[(I)]  $X=\{ \sum_{j=1}^5x_j^2 =\sum_{j=1}^5 a_jx_j^2=0\}$, $a_j$ are distinct, 
\item[(II)] $X=\{ \sum_{j=1}^5x_j^2 =x_1^2 +ax_2^2-x_3^2 -ax_4^2 =0\}$, $a\ne0,\pm1$,
\item[(III)] $X=\{ \sum_{j=1}^5x_j^2 =x_1^2 +\zeta_4x_2^2-x_3^2 -\zeta_4x_4^2 =0\}$, 
\end{itemize}
and the types specialize 
\begin{equation} 
\label{eqn:spec}
\mathrm{(I)} \quad \supsetneq  \quad  \mathrm{(II)} \quad \supsetneq \quad \mathrm{(III)}.
\end{equation}
The corresponding automorphism groups inject: 
$$
C_2^4\subset C_2^4\rtimes C_2\subset C_2^4\rtimes C_4.
$$
Note that actions in types (IV) and (V) are not listed since the
actions of 2-groups on these surfaces are special cases of actions on type (II) surfaces. 
After extracting actions satisfying Condition {\bf (N)} we are left with the groups listed in Table~\ref{table:d4}.  We refer readers to \cite{TZ-support} for the actions not satisfying Condition {\bf (N)} and the specific criteria they violate.

In each type, we only list groups not obtained via specialization in  
\eqref{eqn:spec}, since the geometric properties of the stabilizer stratification of all groups are preserved upon specialization, as will be shown below.  
All actions of 2-groups on surfaces of type (III) satisfying Condition {\bf (N)} are already present in type (II). Thus there are no actions listed as type (III) in Table~\ref{table:d4}.

For the computation of $\Br([X/G])$ we only need to keep track of configurations of curves with nontrivial generic stabilizers. We describe these configurations in detail:
\begin{itemize}
    \item Type (I), $G=C_2^2$, generated by involutions $\iota_{ij},\iota_{rs}$, where $i,j,r,s$ are distinct numbers in $\{1,2,\ldots,5\}$,  $\iota_{ij}$ is the sign change of variables $x_i$ and $x_j$, and $\iota_{rs}$ is defined similarly. There are 15 conjugacy classes of such $C_2^2$ in $\Aut(X)$; note that all 15 actions on $\bP^4$ are conjugate in $\Aut(\bP^4)$. They have the same stabilizer stratification: the only curve with a nontrivial generic stabilizer is the curve
    $$
    E=\{x_t=0\}\cap X,
    $$
    where $t\in\{1,\ldots,5\}\setminus\{i,j,r,s\}$. For any $X$ in type (I), the curve $E$ is a  {\em smooth} curve, of genus 1, with generic stabilizer $C_2=\langle\iota_{ij}\cdot\iota_{rs}\rangle$, and a fixed-point free residual $C_2$-action. The algorithm in \cite{KT-dp} and \cite{PZ} yields
    $$
    \Br([X/G])=(\bZ/2)^2,
    $$
    for all such actions. Upon specialization, the 15 $\Aut(X)$-conjugacy classes collapse to 9 in type (II), and 5 in type (III).  
    \item Type (I), $G=C_2^3$, generated by $\iota_i, \iota_j, \iota_r$ where $i,j,r$ are distinct numbers in $\{1,\ldots,5\}$. There are 10 conjugacy classes of such groups in $\Aut(X)$, with the same stabilizer stratification: three curves have nontrivial generic stabilizers
    $$
    E_i=\{x_i=0\}\cap X,\quad E_j=\{x_j=0\}\cap X,\quad E_r=\{x_r=0\}\cap X.
    $$
    For any $X$ in type (I), each curve is a smooth genus 1 curve, with generic stabilizer $C_2$ and a residual $C_2^2$-action. Taking the quotient by the residual action, we obtain a configuration of 3 rational curves, with pairwise intersections in 2 {\em distinct} points, 
    giving 6 distinct points of intersection, as can be seen by explicit equations, e.g., 
    $$
    E_1\cap E_2=\{x_3^2+x_4^2+x_5^2=a_3x_3^2+a_4x_4^2+a_5x_5^2=0\}\subset\bP^2_{x_3,x_4,x_5}
    $$
    consists of 4 distinct points forming 2 $G$-orbits of length 2, since $a_3,a_4,a_5$ are distinct.
    This yields
    $$
     \Br([X/G])=(\bZ/2)^4. 
    $$
    \item Type (II), $G=C_2^3$, generated by 
     $
    \sigma_{(13)(24)},\iota_{13},\iota_{24},
    $
 where $\sigma_{(13)(24)}$ is the corresponding permutation of the variables. It    
    gives one {\em new} conjugacy class with the following geometry. There are 5 curves with nontrivial generic stabilizers:
    \begin{enumerate}
        \item $E=\{x_5=0\}\cap X$ with stabilizer $C_2=\langle\iota_{13}\cdot\iota_{24}\rangle$.
        \item $Q_1=\{x_1-x_3=x_2-x_4=0\}\cap X$ with stabilizer $C_2=\langle\sigma_{(13)(24)}\rangle$,
        \item $Q_2=\{x_1+x_3=x_2+x_4=0\}\cap X$ with stabilizer $C_2=\langle\sigma_{(13)(24)}\cdot\iota_{13}\cdot\iota_{24}\rangle$,
         \item $Q_3=\{x_1-x_3=x_2+x_4=0\}\cap X$ with stabilizer $C_2=\langle\sigma_{(13)(24)}\cdot\iota_{24}\rangle$,
           \item $Q_4=\{x_1+x_3=x_2-x_4=0\}\cap X$ with stabilizer $C_2=\langle\sigma_{(13)(24)}\cdot\iota_{13}\rangle$.
    \end{enumerate}
    The curves are $G$-invariant, and 
    their intersections are independent of $X$, as is visible from the defining equations of $X$. The quotients by the residual $G$-action on the curves yield a configuration of 5 rational curves, that can be visualized as square with an inscribed circle. We obtain
    $$
\Br([X/G])= (\bZ/2)^4.
    $$    
    \item Type (II), $G=C_2\times C_4$. There are two conjugacy classes of such groups. They are generated by $\sigma_{(13)(24)}\cdot \iota_{i5}$ and $\iota_5$ with $i=3$, and 4, respectively. Both actions have the same stabilizer stratification: there is a unique curve with nontrivial generic stabilizer, 
    $$
    E=\{x_5=0\}\cap X,
    $$
    with stabilizer $C_2=\langle\iota_5\rangle$. For any $X$ in type (II), $E$ is a $G$-invariant smooth genus 1 curve, with a fixed-point free residual $C_4$-action. We obtain 
    $$
    \Br([X/G])=(\bZ/2)^2.
    $$ 
    \item Type (II), $G=C_2\times \mathfrak D_4$.  There are two conjugacy classes of such groups. They are generated by $\sigma_{(13)(24)}, \iota_{5}$ and $\iota_{i}$ with $i=1$, and 2, respectively. Both actions have the same stabilizer stratification. We describe the case when $i=1$ here in detail: there are 7 curves with nontrivial generic stabilizers. They are in the orbits of the curves 
  $$
Q_1=\{x_1+x_3=x_2-x_4=0\}\cap X, \, \, Q_2=\{x_1-x_3=x_2+x_4=0\}\cap X, 
  $$
   $$
E_1=\{x_1=0\}\cap X,\, \,  E_2=\{x_5=0\}\cap X.
  $$
  For any surface $X$ in type (II), each of these curves is a smooth curve with generic stabilizer $C_2$, and has genus 0 after taking quotient by the residual action. As above, the configurations of these curves are independent of $X$. 
  
  After blowing up the orbits of the points 
  $$
  \mathfrak p_1=[0:1:0:1:\zeta_4\sqrt{2}]\quad \text{and}\quad\mathfrak p_2=[0:1:0:-1:\zeta_4\sqrt{2}],
  $$
  the action is in the standard form. After taking quotients by $G$, the configuration of curves with nontrivial generic stabilizer is illustrated as follows:
 $$
  
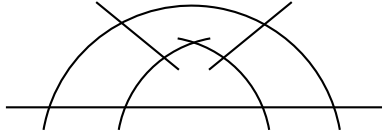
\begin{tikzpicture}
    \draw [thick] (-1,-2) arc[start angle=170, end angle=10, radius=2];
   \draw [thick] (0,-2) arc[start angle=170, end angle=100, radius=1.5];
    \draw [thick] (2,-2) arc[start angle=10, end angle=80, radius=1.5];
    \draw [thick] (-1.5,-1.7) -- (3.5,-1.7);
     \draw [thick] (-0.3,-0.3) -- (0.8,-1.2);
       \draw [thick] (2.3,-0.3) -- (1.2,-1.2);
\end{tikzpicture}
$$
\noindent where the horizontal line comes from $E_2$; the big arc comes from $E_1$; the two small arcs come from $Q_1$ and $Q_2$, the two tilted lines come from the exceptional divisors. By \cite[Corollary 4.6]{KT-Brauer}, this implies that
$$
\Br([X/G])=(\bZ/2)^4.
$$

\end{itemize}

\newpage

\begingroup
\tiny
\setlength{\LTcapwidth}{\textwidth}
\captionsetup{font=tiny}
\begin{longtable}{|c|c|c|c|c|c|c|c|c|c|}
\caption{Del Pezzo surfaces of degree 4}\label{table:d4}\\
\hline
    &$G$ &GapID&
    $\Pic(X)^G$ &$\rH^2(G,k^\times)$&$\Br([X/G])$& $\rH^1(G,\Pic(X))$&$\mathrm{Am}^3$\\\hline
    I & $C_2^2(15)$ & {\tt (4,2)}&$\bZ^2$&$\bZ/2$&$(\bZ/2)^2$&$\bZ/2$&0\\\hline
  I & $C_2^3(10)$ & {\tt (8,5)}&$\bZ$&$(\bZ/2)^3$&$(\bZ/2)^4$&$\bZ/2$&0\\\hline\hline
 
  II & $C_2^3$ & {\tt (8,5)}&$\bZ^2$&$(\bZ/2)^3$&$(\bZ/2)^4$&$\bZ/2$&0\\\hline
II & $C_2\times C_4 (2)$ & {\tt (8,2)}&$\bZ$&$\bZ/2$&$(\bZ/2)^2$&$\bZ/2$&0 \\\hline
II & $C_2\times \fD_4(2)$ & {\tt (16,11)}&$\bZ$&$(\bZ/2)^3$&$(\bZ/2)^4$&$\bZ/2$& 0\\\hline
\end{longtable}
\endgroup

\subsection*{Cubic surfaces} 
In this case, a $p$-group action on $X$ is not $G$-minimal if $p\ne 3$, so we focus on $3$-groups.
We follow \cite[Table 4]{DI} for the classification of the actions. All non-cyclic 3-group actions on cubic surfaces are realized on types (I) and (IV):

\begin{itemize}
    \item[(I)] The Fermat cubic surface $X$,
\item[(IV)] 
 $X=\{ 6ax_2x_3x_4+\sum_{i=1}^4x_i^3=0\}$ where $8a^3+1\ne 0$,
\end{itemize}
and the types specialize 
$$
\mathrm{(IV)}\quad\supsetneq \quad \mathrm{(I)}.
$$

Actions satisfying Condition {\bf (N)} are listed in Table 3. Information about actions not satisfying Condition {\bf (N)} can be found in  \cite{TZ-support}.
We explain the computation of $\Br([X/G])$:

\begin{itemize}
    \item Type (IV), $G=C_3^2$, generated by 
    \begin{align}\label{eqn:iota12}
    \iota_1=\mathrm{diag}(1,1,\zeta_3,\zeta_3^2)\quad\text{and}\quad  \iota_2=\mathrm{diag}(\zeta_3,1,1,1).
    \end{align}
     For any $X$ in type (IV), the only curve with nontrivial generic stabilizer is the smooth genus 1 curve
    $$
    E=\{x_1=0\}\cap X
    $$
   with generic stabilizer $C_3=\langle\iota_2\rangle$, and a fixed-point free residual $C_3$-action. This yields 
   $$
   \Br([X/G])=(\bZ/3)^2.
   $$
   \item Type (IV), $G=C_3^2$, generated by $\iota_2, \sigma_{(234)}\cdot(\iota_1)^i$, where $\iota_1$ and $\iota_2$ are defined as in \eqref{eqn:iota12}, $\sigma_{(234)}$ is the cyclic permutation of variables $x_2,x_3,x_4$, and $i=0,1$ or 2 gives three different conjugacy classes of such actions in $\Aut(X)$. All three conjugacy classes have an identical stabilizer stratification as in the previous case. Thus, for any $X$ in type (IV), we have 
$$
\Br([X/G])=(\bZ/3)^2.
$$
Note that the actions on the ambient $\bP^3$ in this case are conjugate to each other, and conjugate to those in the previous case.
\item Type (I), $G=C_3^3$, generated by 
$$
\mathrm{diag}(\zeta_3,1,1,1),\quad \mathrm{diag}(1,\zeta_3,1,1),\quad \mathrm{diag}(1,1,\zeta_3,1).
$$
The group acts on a unique surface, the Fermat cubic. It has been computed in \cite{KT-dp} and \cite{PZ} that  $\Br([X/G])=(\bZ/3)^3$ in this case.
This yields
\begin{equation} \label{eqn:cube-33}
\mathrm{Am}^3(X,G)=\bZ/3.
\end{equation}
Note that this action fails Condition {\bf (A)}, which already prevents 
the $G$-unirationality of $X$. 
However, \eqref{eqn:cube-33} obstructs projective $G$-unirationality as well.  
 \end{itemize}
\begingroup
\tiny
\setlength{\LTcapwidth}{\textwidth}
\tiny
\captionsetup{font=tiny}
\begin{longtable}{|c|c|c|c|c|c|c|c|}
\caption{Cubic surfaces}\label{table:d3}\\
\hline
     &$G$      & GapID       &$\Pic(X)^G$  &$\rH^2(G,k^\times)$ &$\Br([X/G])$ &$\rH^1(G,\Pic(X))$ & $\mathrm{Am}^3$ \\\hline

IV  & $C_3^2$ & {\tt (9,2)} & $\bZ$       &$\bZ/3$             &$(\bZ/3)^2$  &$\bZ/3$            & $0$\\\hline
IV  & $C_3^2(3)$ & {\tt (9,2)} & $\bZ$       &$\bZ/3$             &$(\bZ/3)^2$  &$\bZ/3$            & $0$ \\\hline\hline 

I    & $C_3^3$ & {\tt (27,5)}& $\bZ$       &$(\bZ/3)^3$         &$(\bZ/3)^3$  &$\bZ/3$            & ${\color{red}{\bZ/3}}$ \\\hline
\end{longtable}
\endgroup

\subsection*{Del Pezzo surfaces of degree 2}
We follow \cite[Table 6]{DI}. Again,  a $p$-group action on $X$ is not $G$-minimal if $p\ne 2$. All 2-group actions satisfying Condition {\bf(N)}  are realized on types (II), (V), (VII) and (X):
\begin{itemize}
    \item[(II)] 
    $
    X=\{w^2=x_1^4+x_2^4+x_3^4\},
    $
    \item[(V)]
    $
X=\{w^2=x^4_1+x_2^4+x_3^4+ax_1^2x_2^2\},
$ where $a\ne \pm2$,
  \item[(VII)]
    $
X=\{w^2=x^4_1+x_2^4+x_3^4+ax_1^2x_2^2+bx_1x_2x_3^2\},
$ where $a\ne \pm2, b^2\ne 4a\pm8$,
  \item[(X)]
    $
X=\{w^2=x^4_1+x_2^4+x_3^4+ax_1^2x_3^2+bx_2^2x_3^2+cx_1^2x_2^2\}
$ where $a,b,c\ne \pm2$ and $a^2+b^2+c^2-abc\ne 4$,
\end{itemize}
and the types specialize 
 $$
\mathrm{(X)} \quad \supsetneq  \quad  \mathrm{(VII)} \quad \supsetneq \quad \mathrm{(V)} \quad \supsetneq \quad \mathrm{(II)},
 $$
with injections of their automorphisms 
$$
C_2^3\subset C_2\times\fD_4\subset C_2\times (\fD_4\rtimes C_2)\subset C_2\times (C_4^2\rtimes C_3 \rtimes C_2).
$$
We introduce the following transformations
\begin{align*}
    \tau: (w,x_1,x_2,x_3)\mapsto (-w,x_1,x_2,x_3),\\
     \sigma: (w,x_1,x_2,x_3)\mapsto (w,x_2,x_1,x_3),\\
      \iota_1: (w,x_1,x_2,x_3)\mapsto (w,-x_1,x_2,x_3),\\
       \iota_2: (w,x_1,x_2,x_3)\mapsto (w,x_1,-x_2,x_3),\\
         \iota_3: (w,x_1,x_2,x_3)\mapsto (w,\zeta_4x_1,\zeta_4^3x_2,x_3),\\
         \iota_4: (w,x_1,x_2,x_3)\mapsto (w,\zeta_4x_1,\zeta_4x_2,x_3),\\
          \iota_5: (w,x_1,x_2,x_3)\mapsto (w,-x_1,\zeta_4x_2,x_3).
\end{align*}

Table~\ref{table:d2} lists actions satisfying Condition {\bf (N)}; we omitted groups with proper subgroups yielding nontrivial $\Am^3(X,G)$, see \cite{TZ-support} for those cases. Also note that when a group acts on surfaces in different types, we only record this action in the most general type, e.g., the $C_2^3$ in the first row of Table~\ref{table:d2} acts on surfaces of all other types, but we only list it once.  

As in the previous cases, we observe that the stabilizer stratification and their configurations are preserved upon specialization. In particular, the computation of $\Br([X/G])$ in each row holds for all surfaces in the corresponding type. We refer readers to \cite{TZ-support} for information about the stabilizer stratification in each case, and actions violating Condition {\bf (N)}.

\begingroup
\tiny
\setlength{\LTcapwidth}{\textwidth}
\tiny 
\captionsetup{font=tiny}
\begin{longtable}{|c|c|c|c|c|c|c|c|c|}
\caption{Del Pezzo surfaces of degree 2}\label{table:d2}\\
\hline
    & $G$ &\!\!GapID\!\!&Generators&$\!\!\Pic(X)^G$\!\!&$\!\!\rH^2(G,k^\times)$\!\!&$\!\!\Br([X/G])$\!\!&$\!\!\rH^1(G,\Pic(X))$\!\!&$\!\!\!\!\mathrm{Am}^3$ \!\!\!\!\\\hline
X   & $C_2^3$ & {\tt (8,5)}&$\iota_1,\iota_2,\tau$&$\bZ$&$(\bZ/2)^3$&$(\bZ/2)^6$&$(\bZ/2)^3$&0\\\hline\hline
VII  & $\fD_4 $ & {\tt (8,3)}&$\sigma,\tau\cdot\iota_3$&$\bZ^2$&$\bZ/2$&$(\bZ/2)^2$&$\bZ/2$&$0$ \\\hline
VII  & $\fD_4 $ & {\tt (8,3)}&$\tau\cdot\sigma,\iota_3$&$\bZ^2$&$\bZ/2$&$(\bZ/2)^2$&$\bZ/2$&$0$ \\\hline
VII  & $\fD_4 $ & {\tt (8,3)}&$\tau\cdot\sigma,\tau\cdot\iota_3$&$\bZ^2$&$\bZ/2$&$(\bZ/2)^2$&$\bZ/2$&$0$ \\\hline
VII   & $C_2^3$ & {\tt (8,5)}&$\tau,\sigma,\iota_3^2$&$\bZ$&$(\bZ/2)^3$&$(\bZ/2)^6$&$(\bZ/2)^3$&0\\\hline
VII   & $C_2^3$ & {\tt (8,5)}&$\tau,\iota_3\cdot\sigma,\iota_3^2$&$\bZ$&$(\bZ/2)^3$&$(\bZ/2)^6$&$(\bZ/2)^3$&0\\\hline
VII  & $C_2\times C_4$ & {\tt (8,2)}&$\tau,\iota_3$&$\bZ$&$\bZ/2$&$(\bZ/2)^3$&$(\bZ/2)^2$&0 \\\hline
VII  & $C_2\times \fD_4$ & {\tt (16,11)}&$\tau,\sigma,\iota_3$&$\bZ$&$(\bZ/2)^3$&$(\bZ/2)^5$&$(\bZ/2)^2$&0\\\hline\hline
V& $ Q_8$ & {\tt (8,4)}&$\tau\cdot\iota_3,\sigma\cdot\iota_1$&$\bZ^2$&$0$&0&$\bZ/2$&{\color{red}{$\bZ/2$}} \\\hline
V& $ Q_8$ & {\tt (8,4)}&$\tau\cdot\iota_3,\tau\cdot\sigma\cdot\iota_1$&$\bZ^2$&$0$&0&$\bZ/2$&{\color{red}{$\bZ/2$}} \\\hline
V& $ Q_8$ & {\tt (8,4)}&$\iota_3,\tau\cdot\sigma\cdot\iota_1$&$\bZ^2$&$0$&0&$\bZ/2$&{\color{red}{$\bZ/2$}} \\\hline

V  & $\fD_4 $ & {\tt (8,3)}&$\tau\cdot\sigma,\tau\cdot\iota_2$&$\bZ^2$&$\bZ/2$&$(\bZ/2)^2$&$\bZ/2$&$0$ \\\hline
V  & $\fD_4 $ & {\tt (8,3)}&$\tau\cdot\iota_3\cdot\sigma,\tau\cdot\iota_2$&$\bZ^2$&$\bZ/2$&$(\bZ/2)^2$&$\bZ/2$&$0$ \\\hline

V  & $\fD_4 $ & {\tt (8,3)}&$\tau\cdot\sigma,\iota_2$&$\bZ^2$&$\bZ/2$&$(\bZ/2)^2$&$\bZ/2$&$0$ \\\hline
V  & $\fD_4 $ & {\tt (8,3)}&$\sigma,\tau\cdot\iota_2$&$\bZ^2$&$\bZ/2$&$(\bZ/2)^2$&$\bZ/2$&$0$ \\\hline
V  & $\fD_4 $ & {\tt (8,3)}&$\tau\cdot\iota_3\cdot\sigma,\iota_2$&$\bZ^2$&$\bZ/2$&$(\bZ/2)^2$&$\bZ/2$&$0$ \\\hline
V  & $\fD_4 $ & {\tt (8,3)}&$\iota_3\cdot\sigma,\tau\cdot\iota_2$&$\bZ^2$&$\bZ/2$&$(\bZ/2)^2$&$\bZ/2$&$0$ \\\hline

V & $C_2\times C_4 $ & {\tt (8,2)}&$\tau\cdot\sigma,\tau\cdot\iota_4$&$\bZ$&$\bZ/2$&$(\bZ/2)^3$&$(\bZ/2)^2$&0 \\\hline

V & $C_2\times C_4 $ & {\tt (8,2)}&$\tau\cdot\iota_2,\iota_3$&$\bZ$&$\bZ/2$&$(\bZ/2)^3$&$(\bZ/2)^2$&0 \\\hline

V & $C_2\times C_4 $ & {\tt (8,2)}&$\tau\cdot\iota_3\cdot\sigma,\sigma\cdot\iota_1$&$\bZ$&$\bZ/2$&$(\bZ/2)^3$&$(\bZ/2)^2$&0 \\\hline

V & $C_2\times C_4 $ & {\tt (8,2)}&$\tau,\sigma\cdot\iota_4$&$\bZ$&$\bZ/2$&$(\bZ/4)^2$&$\bZ/2\oplus\bZ/4$&$0$ \\\hline
V & $C_2\times C_4 $ & {\tt (8,2)}&$\tau,\sigma\cdot\iota_1$&$\bZ$&$\bZ/2$&$(\bZ/4)^2$&$\bZ/2\oplus\bZ/4$&$0$ \\\hline
V & $\fD_4\rtimes C_2 $ & {\tt (16,13)}&$\iota_3,\tau\cdot\sigma,\tau\cdot\iota_2$&$\bZ$&$(\bZ/2)^2$&$(\bZ/2)^2$&$(\bZ/2)^2$&{\color{red}{$(\bZ/2)^2$}}\\\hline
V  & $C_2\times \fD_4$ & {\tt (16,11)}&$\tau,\iota_1,\iota_2,\sigma$&$\bZ$&$(\bZ/2)^3$&$(\bZ/2)^5$&$(\bZ/2)^2$&0\\\hline
V  & $C_2\times \fD_4$ & {\tt (16,11)}&$\tau,\iota_1,\iota_2,\sigma\cdot\iota_3$&$\bZ$&$(\bZ/2)^3$&$(\bZ/2)^5$&$(\bZ/2)^2$&0\\\hline
V  & $C_2^2\times C_4$ & {\tt (16,10)}&$\tau,\iota_1,\iota_3$&$\bZ$&$(\bZ/2)^3$&$(\bZ/2)^4$&$(\bZ/2)^2$&{\color{red}{$\bZ/2$}} 
\\\hline
V  & $C_2^2\times C_4$ & {\tt (16,10)}&$\tau,\sigma\cdot\iota_3,\sigma\cdot\iota_1$&$\bZ$&$(\bZ/2)^3$&$(\bZ/2)^4$&$(\bZ/2)^2$&{\color{red}{$\bZ/2$}} 
\\\hline
V  & $C_2^2\times C_4$ & {\tt (16,10)}&$\tau,\sigma,\iota_4$&$\bZ$&$(\bZ/2)^3$&$(\bZ/2)^4$&$(\bZ/2)^2$&{\color{red}{$\bZ/2$}} 
\\\hline\hline
II  & $\mathrm{OD}_{16} $ & {\tt (16,6)}&$\iota_3,\iota_4,\tau\cdot\sigma\cdot\iota_5$&$\bZ$&$0$&$\bZ/2$&$\bZ/2$&0  \\\hline
\end{longtable}
\endgroup

One can check that Condition {\bf (A)} is satisfied only for the action of  $\mathrm{OD}_{16}$ in Type (II) and the three (nonconjugate) actions of $\mathrm{Q}_8$ in Type (V).

We summarize the analysis:

\begin{theo}\label{theo:dp}
Let $X$ be a del Pezzo surface with a regular generically free action of a finite group $G$. Assume that 
$\mathrm{Am}^3(X,H)\neq 0$, for some $H\subset G$. Then,  up to isomorphism, 
one of the following holds:
\begin{enumerate}
    \item[1.] $X$ is the cubic surface given by 
    $$
    x_1^3+x_2^3+x_3^3+x_4^3=0 
    $$
    and $G$ contains $H=C_3^3$ generated by 
$$
\mathrm{diag}(1,\zeta_3,1,1),\quad \mathrm{diag}(1,1,\zeta_3,1),\quad \mathrm{diag}(1,1,1,\zeta_3).
    $$

    \item[2.] $X$ is a del Pezzo surface of degree 2 given by
$$
w^2=x^4_1+x_2^4+x_3^4+ax_1^2x_2^2, \quad a^2\ne 4,
$$
 and $G$ contains one of the following $H$:
 \begin{enumerate}
     \item $\mathrm{Q}_8$ generated by
$$
    (w,x_1,x_2,x_3)\mapsto (w,-x_2,x_1,x_3),\quad \mathrm{diag}(-1,\zeta_4,\zeta_4^3,1),
    $$
\item
$C_2^2\times C_4$ generated by 
$$
\mathrm{diag}(-1,1,1,1),\quad\mathrm{diag}(1,\zeta_4^3,\zeta_4,1),\quad \mathrm{diag}(1,-1,1,1),
$$

\item $\fD_4\rtimes C_2$ generated by 
\begin{align*}
&(w,x_1,x_2,x_3)\mapsto (-w,x_2,x_1,x_3),\\
&\mathrm{diag}(-1,-1,1,-1),\quad\mathrm{diag}(1,\zeta_4^3,\zeta_4,1).
\end{align*}
 \end{enumerate}
\end{enumerate}
The above actions satisfy $\mathrm{Am}^3(X,H)\neq 0$; in particular, the $G$-actions are not projectively unirational. Moreover, the action in 2(a) is the only ones where Condition {\bf (A)} is satisfied for $H$.  
\end{theo}

\begin{proof}
Assume that $\mathrm{Am}^3(X,G)\ne 0$, then $2\leq K_X^2\leq 4$ by \cite[Proposition 3.4]{MR3359027}.
    When $X$ is $G$-minimal, the assertion follows from the computations and Tables~\ref{table:d4} --~\ref{table:d2} above. When $X$ is not $G$-minimal, we consider the $G$-minimal model $Y$ of $X$. We have $\mathrm{Am}^3(X,G)=\mathrm{Am}^3(Y,G)\ne 0$. Then from Tables~\ref{table:d4} -~\ref{table:d2} again, we know that $2\leq K_Y^2<K_X^2\leq3$. It follows that  $X\to Y$ is a contraction of a $G$-invariant $(-1)$-curve, which implies that $G$ fixes a point on $Y$, contradicting $\mathrm{Am}^3(Y,G)\ne 0$.
\end{proof}

Note that the three actions of $\mathrm{Q}_8$ in Table~\ref{table:d2} are not conjugate in $\Aut(X)$, but their actions on the ambient $\bP(2,1,1,1)$ are isomorphic.  Thus, we may assume the action on the ambient space is given as in Case 2(a), and every invariant degree 4 smooth hypersurface $X$ is given by 
$$
w^2+b_1(x_1^4+x_2^4)+b_2x_3^4+b_3x_1^2x_2^2=0
$$
for some $b_1,b_2,b_3\in k$ and $b_1,b_2\ne 0$. Scaling $x_1,x_2$ and $x_3$, we may assume that $b_1=b_2=1$, and thus $X$ is given as in Case 2(a) above. 
 The same holds for $C_2^2\times C_4$ and Case 2(b).

We explain the computation in these cases in more details. 
The action in Case 2(a) satisfy Condition {\bf (A)} but are not $G$-unirational. This is impossible for del Pezzo surfaces of degree $\geq 3$ by \cite[Theorem 1.4]{Duncan}.

\

\noindent Case 1: This has been addressed in  \cite[Section 5.3]{KT-Brauer}, yielding
$$
\mathrm{Am}^2(X,H)=0,\quad\rH^2(H,k^\times)=(\bZ/3)^3,\quad\rH^1(H,\Pic(X))=\bZ/3.
$$
$$
\Br([X/H])=(\bZ/3)^3, \quad \mathrm{Am}^3(X,H)=\bZ/3.
$$

\

\noindent Case 2(a): In this case, we have 
$$
\rH^2(H,k^\times)=\mathrm{Am}^2(X,H)=0,\quad\rH^1(H,\Pic(X))=\bZ/2,
$$
and the $H$-action on $X$ is in standard form.  
The only divisor with nontrivial generic stabilizer is an $H$-invariant smooth curve $E$ of genus 1, with a generic stabilizer $C_2$. Using Riemann-Hurwitz, we compute the genus $\mathrm{g}(E/H)=0$. It follows that 
$$
\Br([X/H])=0, \quad \mathrm{Am}^3(X,H)=\bZ/2.
$$

\

\noindent Case 2(b): In all cases, we have 
$$
\mathrm{Am}^2(X,H)=0,\quad \rH^2(H,k^\times)=(\bZ/2)^3,\quad\rH^1(H,\Pic(X))=(\bZ/2)^2,
$$
and the $H$-action on $X$ is in standard form. We see from \cite[Section 4.2, Case 2.G24]{PZ} that
$
\Br([X/H])=(\bZ/2)^4,
$
which implies that
$$
\mathrm{Am}^3(X,H)=\bZ/2.
$$

\

\noindent Case 2(c): We have 
$$
\mathrm{Am}^2(X,H)=0,\quad \rH^2(H,k^\times)=(\bZ/2)^2,\quad\rH^1(H,\Pic(X))=(\bZ/2)^2.
$$
We compute the stabilizer stratification:

$$
\begin{tabular}{|c|c|c|c|c|c|}
\hline
& Strata & Stabilizer&Residue&dim&deg\\
\hline
1& $\mathfrak p_1$ &   $\mathrm{Q}_8$& {triv}&  $0$&  1 \\\hline
 2& $\mathfrak p_2$ &   $C_2$& {triv}&  $0$&  1 \\\hline
  3& $\mathfrak p_3$ &   $C_2$& {triv}&  $0$&  1 \\\hline
   4& $\mathfrak p_4$ &   $C_2$& {triv}&  $0$&  1 \\\hline
5&  $E$ &  $C_2$&  $C_2^2$&  $1$&  4\\\hline
\end{tabular}
$$

\ 

\noindent 
The only curve with nontrivial generic stabilizer is an $H$-invariant curve $E$ of genus 1. We have $\mathfrak p_1,\ldots,\mathfrak p_4\not\in E$, and $\mathrm{g}(E/H)=1$. Note that the $H$-action on $X$ is not in standard form. To achieve standard form, we need to blow up the orbit of $\mathfrak p_1$. But since $\mathfrak p_1\not\in E$, the exceptional divisors are rational curves disjoint from $E$, and thus they do not contribute to $\Br([X/H])$.  It follows that 
$$
\Br([X/H])=(\bZ/2)^2,\quad \mathrm{Am}^3(X,H)=(\bZ/2)^2.
$$



\section{Quartic double solids}
\label{sect:solids}

Let $X\to \bP^3$ be a double cover ramified in a quartic surface $S$. 
A {\em very general} nodal $X$ with at most 7 nodes fails stable rationality \cite[Theorem 1.1]{voisin}, see also
\cite{HT-Crelle}. {\em Any} nodal $X$ with at most 6 nodes is irrational; it is rational when 
the number of nodes is at least 11 \cite{chelsol}.

We consider $S$ of Kummer type, i.e., $S=\mathrm{J}(C)/2$, the quotient of the Jacobian of a genus two curve, modulo the standard involution. In this case, $X$ has 16 nodes, the maximal number of nodes on such $X$, in characteristic zero. 
We choose $X$ with maximal $\Aut(X)$, considered in \cite{C-Kummer}. 
They correspond to curves
\begin{align}
\label{eqn:curve}
    C_1&:\{ y^2=x(x^4-1)\},\\
    C_2&:\{y^2=x^5+1\}\notag
\end{align}
and we denote by $X_1$ and $X_2$ the corresponding threefolds. Their automorphisms are known 
$$
\Aut(X_1)= C_2.(C_2^4\rtimes \fS_4),\quad \Aut(X_2)= C_2. (C_2^4\rtimes C_5).
$$
The Sylow subgroups of $\Aut(X_2)$ are $C_2^5$ and $C_5$, which have trivial $\mathrm{B}^3$. Thus, we focus on the first case. Let $X=X_1$. Explicitly, $X$ is given by 
$$
X_1=\{w^2=x_1^4+x_2^4+x_3^4+x_4^4-4\zeta_4x_1x_2x_3x_4\}\subset\bP(2,1,1,1,1), \quad \zeta_4=e^{\frac{2\pi i}{4}},
$$
and $\Aut(X)$ is generated by the sign change on $w$, the $\fS_4$-permutation of $x_1,\ldots,x_4$, and 
\begin{align*}
    (w,x_1,x_2,x_3,x_4)&\mapsto (w,\zeta_4x_1,\zeta_4x_2,-x_3,x_4).
\end{align*}
Next, we compute (via {\tt magma})
\begin{equation} 
\label{eqn:h1}
\rH^1(G,\Pic(\tilde{X})),
\end{equation}
for all subgroups of $\Aut(X)$; here $\tilde{X}$ is the blowup of the 16 nodes of $X$, a smooth model of $X$. We obtain:

\begin{theo} 
\label{thm:kummer}
Let $X\to \bP^3$ be a double cover ramified in $S=    \mathrm{J}(C)/2$, with $C$ given by \eqref{eqn:curve}, and $G\subseteq \Aut(X)$. Assume that the $G$-action on $\tilde{X}$ satisfies Condition {\bf (A)}.
Then 
$$
\text{$\mathrm{Am}^3(\tilde{X},H)\neq 0$ for some $H\subset G$}
$$
if and only if $G$ contains a $\mathrm{Q}_8$ conjugate to one of the following:
\begin{align}\label{eqn:KQ81}
  (w,x_1,x_2,x_3,x_4)\mapsto (w,x_2,x_1,\zeta_4^3x_3,\zeta_4x_4),\, \mathrm{diag}(-1,-1,1,\zeta_4,\zeta_4)
\end{align}
or 
\begin{align}\label{eqn:KQ82}
(w,x_1,x_2,x_3,x_4)&\mapsto (-w,x_2,x_1,\zeta_4x_3,\zeta_4^3x_4),\\
(w,x_1,x_2,x_3,x_4)&\mapsto (-w,x_2,-x_1,-x_3,x_4).\notag
\end{align}
\end{theo}

\begin{proof} 
The proof relies on the algorithm to compute $\Am^3(\tilde{X},G)$ outlined in Section~\ref{sect:coho}. We explain the main steps here and refer the reader to \cite{TZ-support} for computational details.

By \cite{C-Kummer}, the class group $\mathrm{Cl}(X)$ is generated by 32 irreducible surfaces $\Pi_{i}$ which map to 16 planes in $\bP^3$. Let $\tilde X$ be the blowup of the 16 nodes of $X$. There is an exact sequence 
\begin{align}\label{eqn:Picardseq}
    0\to R\to \bigoplus_{i=1}^{48}\bZ\cdot  
D_i\to\Pic(\tilde X)\to 0,
\end{align}
where 
\begin{itemize}
    \item $D_i=\tilde\Pi_i$, $i=1,\ldots,32$, and $\tilde\Pi_i$ is the strict transform of $\Pi$ in $\tilde X$,
    \item $D_{32+i}=E_i$, $i=1,\ldots,16$, and $E_i$ are exceptional divisors above singular points $\mathfrak p_i$ of $X$,
    \item $R$ is the relation space spanned by
    \begin{align}\label{eqn:relation}
         \sum_{i\in I}\left(\tilde\Pi_{i}+\sum_{t\in P_i} E_t\right)-   \sum_{j\in J}\left(\tilde\Pi_{j}+\sum_{s\in P_j} E_s\right)=0,
    \end{align}
  for any $I,J\subset \{1,\ldots,32\}$ such that $|I|=|J|=4$ and 
  $$
  \sum_{i\in I}\Pi_i,\quad \sum_{j\in J}\Pi_j\in|-2K_X|.
  $$
   Each set $P_i\subset \{1,\ldots,16\}$ consists of indices such that $t\in P_i$ if and only if $\mathfrak p_t\in \Pi_i$. 
\end{itemize}

Using this presentation, we find that {\em all} groups $G$ such that
\begin{itemize}
    \item \eqref{eqn:h1} is nontrivial;
    \item Condition {\bf (A)} is satisfied for the $G$-action on $X$;
    \item there is no $G$-fixed smooth point on $X$;
    \item $\mathrm{B}^3(G)$ is nontrivial;
\end{itemize}
contain a subgroup conjugate to one of the two groups $G=\mathrm{Q}_8$ given in \eqref{eqn:KQ81} and \eqref{eqn:KQ82}.

For these two groups, we implement the general formalism explained in Section~\ref{sect:coho} to compute $\Am^3(\tilde{X},G)$. For convenience, the computation uses a periodic resolution of $G=\mathrm Q_8$ given in \cite[Chapter XII.7]{cohomologybook}, to compute cohomology groups with coefficients in a general $G$-module $M$. Choose generators $x,y$ of $G$ such that 
$$
G=\langle x,y\mid x^2=y^2, \quad xyxy^{-1}=1  \rangle.
$$
The group $\rH^i(G,M)$ is the $i$-th cohomology of the periodic complex
\begin{equation}\label{eqn:resolutionQ8}
 M \stackrel{\begin{pmatrix}
  \Delta_x & \!\! \!
  \Delta_y \end{pmatrix}}{\xrightarrow{\hspace*{1.2cm}} }
  M^2 \stackrel{\small\begin{pmatrix}
  L_x& \!\! L_{yx}  \\
 -L_y &\!\!-\Delta_x
  \end{pmatrix}}{\xrightarrow{\hspace*{1.2cm}}}
M^2\stackrel{\begin{pmatrix}
  \Delta_x \\ -\Delta_{yx} \end{pmatrix}}{\xrightarrow{\hspace*{0.8cm}}}
M \stackrel{\begin{pmatrix}
  \Delta_x & \!\! \!
  \Delta_y \end{pmatrix}}{\xrightarrow{\hspace*{1.2cm}} }\cdots
\end{equation}

\noindent 
where $\Delta_x:=1-x$, $L_x:=1+x$, and $\Sigma$ denotes the sum of all group elements in $G$.

Set $D=\cup_i D_i$ and $U=\tilde X\setminus D$, we have an exact sequence 
\begin{align}\label{eqn:liftRseq}
    1\to k^\times \to \bG_m(U)\to R\to 0.
\end{align}
Up to scalar, each element in $R$ of the form \eqref{eqn:relation} lifts to a function 
$
f_I/f_J\in \bG_m(U),
$ 
where $f_I$ and $f_J$ are degree-2 polynomials (on $X$) in the variables $w,x,y,z$, with weights $(2,1,1,1)$, such that
$$
\sum_{i\in I}\Pi_i=\{f_I=0\}\cap X,\quad \sum_{j\in J}\Pi_j=\{f_J=0\}\cap X.
$$

To determine the image of $\delta_3$, we use the following commutative diagram mentioned in Section~\ref{sect:coho}
\[
\xymatrix@C=12pt@R=15pt{
&&
\rH^2(G, \mathbb G_m(U))\ar[d] & \\
0 \ar[r] &
\rH^1(G,\Pic(\tilde X)) \ar[r]\ar[dr]_{\delta_3} &
\rH^2(G, R) \ar[r] \ar[d] &
\rH^2(G,\displaystyle{\bigoplus_{i=1}^{48} \bZ\cdot D_{i}}) \\
&&\rH^3(G,k^\times)&
}
\]

First, we consider $G=\mathrm{Q_8}$ given by \eqref{eqn:KQ81}. Using \eqref{eqn:Picardseq}, we find that
\begin{align}\label{eqn:H1Q81}
    \rH^1(G,\Pic(\tilde X))=(\bZ/2)^2.
\end{align}
Combining \eqref{eqn:resolutionQ8} and \eqref{eqn:Picardseq}, we identify \eqref{eqn:H1Q81} as elements in $R^2$, which can then be lifted to elements in $\bG_m(U)^2$ via \eqref{eqn:liftRseq}. Applying the differential $\begin{pmatrix}
    \Delta_x &\!\!\!-\Delta_{xy}
\end{pmatrix}$ in \eqref{eqn:resolutionQ8}, we find that two elements of \eqref{eqn:H1Q81} map to $-1\in\bG_m(U)$ and the other two elements map to $1\in\bG_m(U)$. Since the $G$-action on $X$ satisfies Condition {\bf (A)} and $\mathrm B^3(G)=\bZ/2$, we know that
$$
\mathrm{Am}^3(\tilde X,G)=\bZ/2.
$$

When $G=\mathrm{Q_8}$ is given by \eqref{eqn:KQ82}, the same steps as above yield 
$$
\rH^1(G,\Pic(\tilde X))=\mathrm{Am}^3(\tilde X,G)=\bZ/2.
$$
\end{proof}

\begin{rema}
The computation in the proof of Theorem~\ref{thm:kummer} suggests that
$$
\Br([\tilde{X}/G])=
\begin{cases}
    \bZ/2& \text{ when } G \text{ is given by }\eqref{eqn:KQ81},\\
    0&\text{ when }  G \text{ is given by }\eqref{eqn:KQ82}.\\
\end{cases}
$$
We display the fixed locus stratification of $X$ for the action of \eqref{eqn:KQ81}:
$$
\begin{tabular}{|c|c|c|c|c|c|}
\hline
& Strata & Stabilizer&Residue&dim&deg\\
\hline
 1--2& $\mathfrak q_i$ &   $C_4$& {triv}&  $0$&  1 \\\hline
3--4& $\mathfrak q_i$ &   $C_4$& {triv}&  $0$&  1 \\\hline
5&  $E_1$ &  $C_4$&  $C_2$&  $1$&  4\\\hline
6&  $E_2$ &  $C_2$&  $C_2^2$&  $1$&  4\\\hline
\end{tabular}
$$
and  that of \eqref{eqn:KQ82}:
$$
\begin{tabular}{|c|c|c|c|c|c|}
\hline
& Strata & Stabilizer&Residue&dim&deg\\
\hline
 1--2& $\mathfrak q_i$ &   $C_4$& {triv}&  $0$&  1 \\\hline
3--4& $\mathfrak q_i$ &   $C_4$& {triv}&  $0$&  1 \\\hline
5--8& $\mathfrak q_i$ &   $C_4$& {triv}&  $0$&  1 \\\hline
9&  $E_1$ &  $C_2$&  $C_2^2$&  $1$&  4\\\hline
10&  $E_2$ &  $C_2$&  $C_2^2$&  $1$&  4\\\hline
\end{tabular}
$$
Note that, unlike in the case of surfaces, we obtain a nontrivial Brauer group of the quotient stack despite the absence of divisors with nontrivial generic stabilizers. This suggests that one needs to blow up additional strata to see divisors that contribute to $\Br([\tilde{X}/G])$, see \cite[Section 4 and 8]{KT-Brauer} for another such instance and an explanation.
\end{rema}

\begin{rema} 
Considering the $G$-invariant hyperplane section of $X$ given by 
$x_4=0$ in \eqref{eqn:KQ81} and \eqref{eqn:KQ82}, 
we recover the $G=\mathrm{Q_8}$-actions on the degree 2 del Pezzo surface from Case 2(a) in Theorem~\ref{theo:dp}. Applying 
Proposition~\ref{prop:obstr}, we obtain 
another proof that these surfaces are not $G$-unirational.  
\end{rema}

\bibliographystyle{plain}
\bibliography{delta3}

\end{document}